\documentclass{amsart}  

\usepackage{amsfonts,amssymb,euscript}
\usepackage{mathrsfs,amsmath,amsthm}

\DeclareMathAlphabet\eufrak{U}{euf}{m}{n}
\SetMathAlphabet\eufrak{bold}{U}{euf}{b}{n}

\theoremstyle{plain}
\newtheorem{theorem}{Theorem}[section]
\newtheorem{proposition}[theorem]{Proposition}
\newtheorem{lemma}[theorem]{Lemma}

\theoremstyle{definition}
\newtheorem{definition}[theorem]{Definition}

\theoremstyle{remark}
\newtheorem{remark}[theorem]{Remark}

\begin{document}

\title{L\'evy processes on the Lorentz-Lie algebra}

\author{Ameur Dhahri}
\address{Dipartimento di Matematica, Politecnico di Milano,
Piazza Leonardo da Vinci 32, I-20133 Milano, Italy}
\email{ameur.dhahri@polimi.it}

\author{Uwe Franz}
\address{Laboratoire de math\'ematiques de Besan\c{c}on,
Universit\'e de Bourgogne Franche-Comt\'e,
16, route de Gray, 25 030 Besan\c{o}n cedex, France}
\email{uwe.franz@univ-fcomte.fr}

\begin{abstract}
L\'evy processes in the sense of Sch\"urmann on the Lie algebra of the Lorentz grouop are studied. It is known that only one of the irreducible unitary representations of the Lorentz group admits a non-trivial one-cocycle. A Sch\"urmann triple is constructed for this cocycle and the properties of the associated L\'evy process are investigated. The decommpositions of the restrictions of this triple to the Lie subalgebras $so(3)$ and $so(2,1)$ are described.
\end{abstract}

\maketitle

\section*{Introduction}

Factorisable representations of current groups and current algebras have a long history, cf. \cite{araki69/70,parthasarathy+schmidt72,guichardet72}, and they played an important role in the development of quantum stochastic calculus, see \cite{streater00} and the references therein.

Factorisable representations of current groups of a Lie group $G$ can be viewed as L\'evy processes in the sense of Sch\"urmann \cite{schuermann93} on the level of the associated universal enveloping algebra $U(\eufrak{g})$ or the Lie algebra $\eufrak{g}$, see \cite{franz04,franz+privault16}. We followed this approach in \cite{accardi+franz+skeide02}, where we associated classical L\'evy processes to the representation introduced in \cite{vershik+al73} and studied, e.g., their marginal distributions. Furthermore, in \cite{accardi+dhahri09a,accardi+dhahri09b}, this approach was used to define ``quadratic'' exponential vectors and a ``quadratic'' second quantization functor.

Interesting factorisable representations of current groups of a Lie group exist only if the Lie group has a representation which admit a non-trivial cocycle, see \cite{vershik+al74,vershik+karpushev82,graev+vershik05}. If we restrict our attention to unitary representations and simple Lie groups, then this leaves only the two series. $G=SO(n,1)$ and $G=SU(n,1)$. These are the simple Lie groups which do not have Kazhdan's property (T). According to Graev and Vershik  \cite{graev+vershik05,vershik+graev06}, this fact was first observed in 1973-74 by Gelfand, Graev, and Vershik \cite{vershik+al74} (but we where not able to confirm this claim from the paper's English translation), according to Shalom \cite{shalom00} this was first proved by Delorme \cite{delorme77} and Hotta and Wallach \cite{hotta+wallach75}.

In \cite{accardi+franz+skeide02}, we focussed on the lowest-dimensional case, i.e.\ the unique non-compact form $so(2,1)\cong su(1,1)\cong sl(2,\mathbb{R})$ of the unique simple Lie algebra of rank one. In the present paper we study the Lie algebra $so(3,1)$ of the Lorentz group.

In Section \ref{sec-lorentzgroup}, we recall the definition of the Lorentz group and its Lie algebra and we introduce some notations which we shall use in this paper. For the purpose of self-containedness we also recall several facts about their unitary representations.

In Section \ref{sec-levy},  we recall the definitions of L\'evy processes and Sch\"urmann triples on real Lie algebras. We also construct a Sch\"urmann triple on $so(3,1)$, which has a non-trivial 1-cocycle. From the general theory mentioned above, it is clear that this Sch\"urmann triple is unique up to rescaling and adding a coboundary. Therefore the L\'evy process associated to this Sch\"urmann triple must be the infinitesimal version of the factorisable representations on $SO(3,1)$ studied in \cite{graev+vershik05,vershik+graev06}.

In the remaining sections we study restrictions of this Sch\"urmann triple to Lie subalgebras of $so(3,1)$ In Section \ref{sec-restr-so3}, we obtain the decomposition of the restriction to $so(3)$. In Section \ref{sec-restr-so21}, we study the restriction to $so(2,1)$.

\section{The Lorentz group and its Lie algebra}\label{sec-lorentzgroup}

The \emph{Lorentz group} $O(3,1)$ is the the group of all isometries of Minkowski spacetime, i.e., it is the group of all $4\times 4$ matrices that leave invariant the Minkowksi inner product
\[
\left\langle \left(\begin{array}{c} t \\ x \\ y \\ z \end{array}\right), \left(\begin{array}{c} t' \\ x' \\ y' \\ z' \end{array}\right) \right\rangle = tt'-xx'-yy'-zz'.
\]
The identity component $SO(3,1)^+$ of $O(3,1)$ is called the \emph{restricted Lorentz group}. It consists of the $4\times 4$ matrices $A=(a_{jk})\in O(3,1)$ with $\det(A)=+1$ and $a_{11}\ge 1$.

The restricted Lorentz group $SO(3,1)^+$ is isomorphic to the projective special linear group (or M\"obius group) $PSL(2,\mathbb{C})$.

The Lie algebra $so(3,1)\cong sl(2,\mathbb{C})$ of the Lorentz group $O(3,1)$ has as basis $\{H_1,H_2,H_3,F_1,F_2,F_3\}$ with the relations
\begin{gather*}
[H_j,H_k]=i\epsilon_{jk\ell} H_\ell. \qquad [F_j,F_k]= - i\epsilon_{jk\ell} F_\ell, \\
[F_j,H_k] = i \epsilon_{jk\ell} H_\ell,
\end{gather*}
for $j,k,\ell\in\{1,2,3\}$, where $\epsilon_{jk\ell}$ is the Levi-Civita symbol,
\[
\epsilon_{j,k,\ell} = \left\{
\begin{array}{cll}
+1 & \mbox{ if } & (jk\ell)=(1,2,3), (2,3,1)\mbox{ or }(3,1,2), \\
-1 & \mbox{ if } & (jk\ell)=(2,1,3), (3,2,1)\mbox{ or }(1,3,2), \\
0 & \mbox{else}.&
\end{array}
\right.
\]
We shall consider $so(3,1)$ with the involution which makes these six elements hermitian. Then the infinitesimals of unitary representations of $SO(3,1)^+$ are *-representations of $so(3,1)$. For our computations we will also use the bases $\{A_1,A_2,A_3,B_1,B_2,B_3\}$ and $\{H_3,H_-,H_+,K_3,K_-,K_+\}$ with
\begin{gather*}
A_j = \frac{1}{2} (H_j+iF_j),\quad B_j = \frac{1}{2}(H_j-i F_j), \quad j=1,2,3,\\
H_3, \quad H_{\pm} = H_1\pm i H_2, \quad F_3, \quad K_{\pm} = K_1 \pm i K_2.
\end{gather*}
In terms of these bases the relations become
\begin{gather*}
[A_j,A_k]=i\epsilon_{jk\ell} A_\ell, \quad [B_j,B_k]=i\epsilon_{jk\ell}B_\ell, \quad [A_j,B_k] = 0 \\
[H_+,H_-] = 2 H_3, \quad [F_+,F_-] = - 2 H_3, \\
[H_3,H_\pm] = \pm H_\pm, \quad [F_3,F_\pm] = \mp H_\pm, \\
[H_3, F_\pm] = \pm F_\pm, \quad [F_3,H_\pm] = \pm F_\pm, \\
[H_\pm,F_\mp]=\pm F_3,  \quad [H_\pm,F_\pm] = 0,\quad [F_3,H_3] = 0,
\end{gather*}
and
\[
A_j^* = B_j,\quad H_\pm^* = H_\mp, \quad F_\pm^* = F_\mp.
\]
The representation theory of $SO(3,1)^+$ can be found in the monographs \cite{gelfand+al63,naimark64,ruehl70}. Here we are only interested in unitary representations. On the level of the Lie algebra $so(3,1)$ they can be described as follows. For $\ell_0 \in \frac{1}{2}\mathbb{Z}$, $\ell_0\ge 0$, and $\ell_1\in\mathbb{C}$, let
\[
D_{\ell_0\ell_1} = \mathrm{span}\{\xi_{\ell m}; \ell=\ell_0,\ell_0+1,\ldots,m=-\ell,-\ell+1,\ldots,\ell\}
\]
and set
\begin{eqnarray*}
C_\ell &=& \left\{
\begin{array}{ll}
i\frac{\sqrt{(\ell^2-\ell_0^2)(\ell^2-\ell_1^2)}}{\ell \sqrt{4\ell^2-1}} & \mbox{ if } \ell\ge 1, \\
0 & \mbox{ if } \ell =0, \frac{1}{2},
\end{array}\right. \\
A_\ell &=& \left\{
\begin{array}{ll} \frac{i \ell_0\ell_1}{\ell(\ell+1)}& \mbox{ if } \ell>0, \\
0 & \mbox{ if } \ell =0,
\end{array}\right.
\end{eqnarray*}
(note that $\ell=0$ or $\ell=\frac{1}{2}$ can only occur if $\ell_0=0$ or $\ell_0=\frac{1}{2}$, resp.).

We define an action of $so(3,1)$ on $D_{\ell_0\ell_1}$ by
\begin{gather*}
\rho_{\ell_0\ell_1}(H_3) \xi_{\ell m} = m \xi_{\ell m}, \quad \rho_{\ell_0\ell_1}(H_\pm ) \xi_{\ell m} = \sqrt{(\ell\mp m)(\ell \pm m+1)}\xi_{\ell,m\pm1}, \\
\rho_{\ell_0\ell_1}(F_3)\xi_{\ell m} =  C_\ell \sqrt{\ell^2-m^2} \xi_{\ell-1,m}
- m A_\ell \xi_{\ell m} \\
-C_{\ell+1} \sqrt{(\ell+1)^2-m^2} \xi_{\ell+1,m}, \\
\rho_{\ell_0\ell_1}(F_\pm) \xi_{\ell m} = \pm C_\ell \sqrt{(\ell\mp m)(\ell \mp m-1)}\xi_{\ell-1,m\pm 1} \\
- A_\ell \sqrt{(\ell \mp m)(\ell\pm m +1)}\xi_{\ell,m\pm 1}
\pm C_{\ell+1} \sqrt{(\ell \pm m +1)(\ell\pm m +2)} \xi_{\ell+1,m\pm 1}
\end{gather*}
(we use the same basis as \cite{gelfand+al63} and \cite{naimark64}, note that \cite{atakishiyev+suslov85} uses $\psi_{\ell m}=i^{m-\ell}\xi_{\ell m}$ instead).

As was observed in \cite{atakishiyev+suslov85}, ``formally'' there exists also a basis $\{\phi_{m_1,m_2}\}$ on which $A_3$, $A_\pm=A_1\pm i A_2$, $B_3,$ and $B_\pm=B_1\pm i B_2$ act as
\begin{gather*}
A_3 \phi_{m_1,m_2} = m_1 \phi_{m_1,m_2}, \quad A_\pm = \sqrt{(j_1\mp m_1)(j_1\pm m_1+1)} \phi_{m_1\pm1,m_2}, \\
B_3 \phi_{m_1,m_2} = m_2 \phi_{m_1,m_2}, \quad B_\pm = \sqrt{(j_2\mp m_2)(j_2\pm m_2+1)} \phi_{m_1,m_2\pm1},
\end{gather*}
where $j_1=\overline{j_2} = \frac{1}{2}(\ell_0+\ell_1-1)$.

Define an inner product on $D_{\ell_0\ell_1}$ s.t.\ the family $\{\xi_{\ell m}\}$ is an orthonormal system and denote by $H_{\ell_0\ell_1}$ the completion of $D_{\ell_0\ell_1}$.

The representation $\rho_{\ell_0\ell_1}$ is the infinitesimal representation of an irreducible unitary representation of $SO(3,1)^+$ on $H_{\ell_0\ell_1}$ in the following two cases:
\begin{description}
\item[a)]
$\ell_1$ is purely imaginary (and no restriction on $\ell_0\in \frac{1}{2}\mathbb{Z}_+$), this is the \emph{principal series}.
\item[b)]
$\ell_0=0$ and $0\le \ell_1<1$, this is the \emph{supplementary series}.
\end{description}
The representations $\rho_{0,\ell_1}$ and $\rho_{0,-\ell_1}$ with $\ell_1$ purely imaginary are easily seen to be equal, since only the square of $\ell_1$ occurs in their definition. The remaining representations are all inequivalent. Together with the trivial representation $\varepsilon$, which sends $so(3,1)$ identically to $0$, these two families exhaust all irreducible unitary representations of $SO(3,1)^+$. Note that the representation $\rho_{01}$ is not irreducible. It can be decomposed as $\rho_{01}\cong \varepsilon\oplus\rho_{10}$ on
\[
D_{01} = \mathrm{span}\{\xi_{00}\}\oplus\mathrm{span}\{\xi_{\ell m}; \ell\ge 1, m=-\ell,\ldots,\ell\}.
\]

The elements $C_A=2A_3^2 + A_+A_-+A_-A_+$ and $C_B=2B_3^2+B_+B_-+B_-B_+$ generate the center of the universal enveloping algebra $U\big(so(3,1)\big)$ and satisfy $C_A^*=C_B$.
The Casimir invariants
\begin{gather*}
J_1 =  C_A+C_B = 2H_3^2 +H_+H_-+H_-H_+ - 2 F_3^2 - F_+F_- - F_-F_+,\\
J_2 = -i(C_A-C_B)/2 = H_+F_- + H_-F_++F_+H_-+F_-H_+ + 4F_3H_3,
\end{gather*}
are hermitian and also generate the center of $U\big(so(3,1)\big)$. In the irreducible unitary representations defined above they act as
\[
\rho_{\ell_0\ell_1}(J_1) = (\ell_0^2+\ell_1^2-1)\mathrm{id}_{D_{\ell_0\ell_1}} \quad \mbox{ and } \quad \rho_{\ell_0\ell_1}(J_2) = i\ell_0\ell_1\mathrm{id}_{D_{\ell_0\ell_1}}.
\]
In this paper we will work with (in general) unbounded involutive representations of involutive complex Lie algebras $(\eufrak{g},*)$ or their universal envelopping algebras $U(\eufrak{g})$ acting on some pre-Hilbert space $D$. See \cite{nelson} for necessary and sufficient conditions for such a representationation to be the infinitesimal representation associated to a unitary representation $H=\overline{D}$ of the connected simply connected Lig group $G$ associated to the real Lie algebra $\eufrak{g}_\mathbb{R}=\{X\in\eufrak{g}; X^*=-X\}$.

\section{Sch\"urmann triples on $so(3,1)$ and their L\'evy processes}\label{sec-levy}

We start by recalling the definition of Sch\"urmann triples and L\'evy processes on real Lie algebras. Let $\eufrak{g}$ be a real Lie algebra, ($\eufrak{g}_\mathbb{C},*)$ its complexification equipped with the involution that makes the elements of $\eufrak{g}$ anti-hermitian, and $U_0(\eufrak{g})$ its universal enveloping algebra, without unit, but with the involution induced from $(\eufrak{g}_\mathbb{C},*)$

For $D$ a complex pre-Hilbert space with, we let
$\mathcal{L}(D)$ be algebra of linear operators on $D$ having an
adjoint defined everywhere on $D$, and $\mathcal{L}_{\rm AH}(D)$ the anti-Hermitian linear operators on $D$.

\begin{definition}
\label{def levy}
\cite[Definition 8.1.1]{franz+privault16}
Let $D$ be a pre-Hilbert space and $\omega\in D$ a unit vector. A family
\[
 \big(j_{st} :  \eufrak{g} \to \mathcal{L}_{\rm AH}(D)\big)_{0\le s\le t}
\]
 of representations of $\eufrak{g}$ is a \emph{L\'evy process on $\eufrak{g}$ over $(D,\omega)$} if
\begin{description}
\item[a)]
(increment property) we have
\[
j_{st}(X)+j_{tu}(X)=j_{su}(X)
\]
for all $0\le s\le t \le u$ and all $X\in\eufrak{g}$;
\item[b)]
(independence)
 we have
\[
[j_{st}(X),j_{s't'}(Y)]=0, \qquad X,Y\in\eufrak{g},
\]
 $0\le s\le t\le s'\le t'$, and
\begin{eqnarray*}
\langle\omega , j_{s_1t_1}(X_1)^{k_1}\cdots j_{s_nt_n}(X_n)^{k_n}\omega\rangle
 & = & \langle\omega , j_{s_1t_1}(X_1)^{k_1}\omega\rangle\cdots\langle\omega , j_{s_nt_n}(X_n)^{k_n}\omega \rangle,
\end{eqnarray*}
 for $n,k_1,\ldots,k_n\in\mathbb{N}$,
 $0\le s_1\le t_1\le s_2\le \cdots \le t_n$, $X_1,\ldots,X_n\in\eufrak{g}$;
\item[c)]
(stationarity)
for $n\in\mathbb{N}$, $X\in\eufrak{g}$, $0\le s\le t$,
\[
\langle\omega ,j_{st}(X)^n\omega\rangle = \langle\omega ,j_{0,t-s}(X)^n\omega\rangle
\]
i.e., the moments depend only on the difference $t-s$;
\item[d)]
(pointwise continuity)
we have
\[
 \lim_{t\searrow s} \langle\omega ,j_{st}(X)^n\omega\rangle =0,
 \qquad
 n\in\mathbb{N}, \quad
 X\in\eufrak{g}.
\]
\end{description}
\end{definition}
To a L\'evy process we can associate the states $\varphi_t=\langle\omega ,j_{0,t}(a)^n\omega\rangle$ for $a\in U_0(\eufrak{g})$, $t\ge 0$, and the functional
\[
L(a) = \left.\frac{{\rm d}}{{\rm d}t}\right|_{t=0} \varphi_t(a), \qquad a\in U_0(\eufrak{g}.
\]
The functional $L$ is a generating functional in the sense of the following definition: A linear functional $L:U_0(\eufrak{g})\to \mathbb{C}$ on the non-unital *-algebra $U_0(\eufrak{g})$ is called a \emph{generating functional on $\eufrak{g}$} if
\begin{description}
\item[a)]
$L$ is Hermitian, i.e., $L(u^*)=\overline{L(u)}$ for $u\in U_0(\eufrak{g})$;
\item[b)]
$L$ is positive, i.e., $L(u^*u)\ge 0$ for $u\in U_0(\eufrak{g})$.
\end{description}
\begin{definition}\label{def triple}
\cite[Definition 8.2.1]{franz+privault16}
Let $D$ be a pre-Hilbert space. A \emph{Sch\"urmann triple on $\eufrak{g}$ over $D$} is a triple $(\rho,\eta,\psi)$, where
\begin{description}
\item[a)]
$\rho : \eufrak{g}\to \mathcal{L}_{\rm AH}(D)$ is a representation of $\eufrak{g}$ on $D$, i.e.,
\[
\rho\big([X,Y]\big)=\rho(X)\rho(Y)-\rho(Y)\rho(X)
\]
for $X,Y\in\eufrak{g}$;
\item[b)]
$\eta : \eufrak{g}\to D$ is a $\rho$-$1$-cocycle, i.e.\ it satisfies
\[
\eta\big([X,Y])=\rho(X)\eta(Y)-\rho(X)\eta(Y), \qquad X,Y\in\eufrak{g};
\]
\item[c)]
$\psi : \eufrak{g}\to \mathbb{C}$ is a linear functional with imaginary values s.t.\ the bilinear map $(X,Y)\longmapsto \langle\eta(X),\eta(Y)\rangle$ is the 2-$\varepsilon$-$\varepsilon$-coboundary of $\psi$ (where $\varepsilon$ denotes the trivial representation), i.e.,
\[
\psi\big([X,Y]\big)=\langle\eta(Y),\eta(X)\rangle-\langle\eta(X),\eta(Y)\rangle,
 \qquad
  X,Y\in\eufrak{g}.
\]
\end{description}
\end{definition}
See \cite{guichardet80} for more information on the cohomology of Lie algebras and Lie groups.

The Sch\"urmann triple and in particular the linear functional $\psi$ in a Sch\"urmann triple have unique extensions to $\eufrak{g}_\mathbb{C}$ (by linearity) and to $U_0(\eufrak{g})$ (as representation, cocycle and coboundary, resp.) and the extension of the functional is a generating functional. Therefore it corresponds to a L\'evy process on $\eufrak{g}$ (which is unique in distribution). This L\'evy process can be realised on the symmetric Fock space
$\Gamma\big(L^2(\mathbb{R}_+,D)\big)=\bigoplus_{n=0}^\infty L^2(\mathbb{R}_+,D)^{\otimes_s n}$
over $L^2(\mathbb{R}_+,D)$ as
\[
j_{st}(X)=\Lambda_{st}\big(\rho(X)\big)+A^+_{st}\big(\eta(X)\big)+A^-_{st}\big(\eta(X^*)\big)+\psi(X)(t-s){\rm id},\qquad X\in\eufrak{g}_\mathbb{C},
\]
cf.\ \cite{schuermann93}. Here $\Lambda_{st}(M)=\Lambda(M\otimes \mathbf{1}_{[s,t]})$, $A^+_{st}(v)=A(v\otimes \mathbf{1}_{[s,t]})$, and $A^-_{st}(v)=A(v\otimes \mathbf{1}_{[s,t]})$, with $M\in\mathcal{L}(D)$, $v\in D$, denote the conservation, creation, and annihilation operators, see, e.g., \cite{franz+privault16}, Chapter 5].

If the cocycle $\eta$ is a coboundary, then the associated L\'evy process is unitarily equivalent to the second quantisation of $\rho$, cf.\ \cite[Proposition 8.2.7]{franz+privault16}, and the L\'evy processes of cocycles that differ only by a coboundary are also unitarily equivalent. Therefore it is most interesting to study the processes associated to non-trivial cocycles. The L\'evy processes associated to reducible Sch\"urmann triples can be constructed as tensor products of the L\'evy processes of their irreducible components, for this reason we shall study only irreducible representations.

If the Casimir invariants are invertible in some representation, then all 1-cocyles are coboundaries, cf.\ \cite[Lemma 2.2]{accardi+franz+skeide02}. Therefore the only non-trivial irreducible unitary representation that can have a non-trivial cocycle is $\rho_{10}$. Since $so(3,1)$ is simple, it is clear that the trivial representation has no non-zero cocycles at all, cf.\ \cite[Lemma 2.1]{accardi+franz+skeide02}.

From \cite{delorme77,hotta+wallach75}, we know that only one irreducible unitary representation of $SO(3,1)$ admits a non-trivial 1-cocycle, and this is $\rho_{10}$. We will describe a non-trivial 1-cocycle of this representation below, after recalling a useful lemma.

\begin{lemma} (Raabe-Duhamel test)
Let $(u_n)_{n\in\mathbb{N}}$ be a sequence of strictly positive real numbers such that
\[
\frac{u_{n+1}}{u_n} = 1 - \frac{\alpha}{n}+o\left(\frac{1}{n}\right).
\]
If $\alpha<1$, then the series $\sum_{n\in\mathbb{N}} u_n$ diverges, if $\alpha>1$, then the series $\sum_{n\in\mathbb{N}} u_n$ converges (nothing can be concluded in the case $\alpha=1$).
\end{lemma}

We can now construct a non-trivial 1-cocycle for $\rho_{10}$.

\begin{proposition}
There exists a $1$-$\rho_{10}$-cocycle $c$ with $c(F_+)=\xi_{11}$. The cocycle $c$ is not a coboundary and every other non-trivial $1$-$\rho_{10}$-cocycle is a linear combination of $c$ and some $1$-$\rho_{10}$-coboundary.
\end{proposition}
\begin{proof}
The cocycle $c$ can not be a coboundary, since the vector $\xi_{11}$ is not in the image of $\rho_{10}(F_+)$. Indeed, assume there exists a vector $\zeta=\sum x_{\ell m} \xi_{\ell m}\in H_{10}$ s.t.\ $\rho_{10}(F_+)\zeta=\xi_{11}$. Then we have $x_{\ell m}=0$ for all $m\not=0$ and for pairs with $m=0$ and $\ell$ odd. For $m=0$ and even $\ell$ we find the recurrence relation
\[
x_{\ell,0} = -\frac{C_{\ell-1}}{C_\ell}x_{\ell-2,0}
\]
with the initial condition $x_{20}=-i\sqrt{\frac{5}{2}}$. We have
\begin{eqnarray*}
\frac{|x_{\ell,0}|^2}{|x_{\ell-2,0}|^2} &=& \frac{\big((\ell-1)^2-1\big)(4\ell^2-1)}{(\ell^2-1)\big(4(\ell-1)^2-1\big)} = \frac{(\ell^2-2\ell)(4\ell^2-1)}{(\ell^2-1)(4\ell^2-8\ell+3)} \\
&=& \frac{4\ell^4-8\ell^3-\ell^2+2\ell}{4\ell^4-8\ell^3-\ell^2+8\ell+3} = 1 + o\left(\frac{1}{\ell}\right)
\end{eqnarray*}
i.e., $\alpha=0<1$. So the Raabe-Duhamel test implies that the series $\sum |x_{\ell m}|^2$ diverges and therefore there exists no such vector $\zeta$ in $H_{10}$.

It was shown in \cite{delorme77,hotta+wallach75} that the first cohomology group of $\rho_{10}$ has dimension one, this implies the uniqueness. To prove existence, one checks that
\begin{gather*}
c(H_3) = c(H_\pm) = 0 \\
c(F_3) = -\frac{1}{\sqrt{2}} \xi_{10}, \quad c(F_\pm) = \pm \xi_{1,\pm1}
\end{gather*}
defines indeed a $1$-$\rho_{10}$-cocycle.
\end{proof}

\begin{remark}
Let $\ell_0=0$. The $1$-$\rho_{0,\ell_1}$-coboundary of $\xi_{00}$ is given by
\begin{gather*}
\partial \xi_{00} (H_3) = 0 = \partial \xi_{00}(H_\pm), \\
\partial \xi_{00}(F_3) = -iC_1(0,\ell_1)\xi_{10}, \quad \partial \xi_{00}(F_\pm) = \pm i C_1(0,\ell_1)\sqrt{2}\xi_{1,\pm1},
\end{gather*}
with $C_1(0,\ell_1)= \frac{\sqrt{1-\ell_1^2}}{\sqrt{3}}$.
We see that $c$ is formally the limit of the $1$-$\rho_{0,\ell_1}$-coboundaries $\frac{1}{i C_1(0,\ell_1)\sqrt{2}}\partial \xi_{00}$ as $\ell_1 $ tends to $1$.
\end{remark}

\begin{proposition}\label{prop c}
There exists a unique Sch\"urmann triple $(\rho_{10},c,\psi)$ containing the representation $\rho_{10}$ and the cocycle $c$.
\end{proposition}
\begin{proof}
This is a consequence of the fact that $so(3,1)$ is simple and that therefore the second cohomology group of the trivial representation is trivial. Any element in $so(3,1)$ can be written as a commutator and so Condition c) in Definition \ref{def triple} allows to deduce the values of $\psi$ from the values of $c$. We have, e.g.,
\[
2\psi(F_3) = \psi\big([F_+,H_-]\big) = \Big\langle c\big((F^+)^*\big),c(H_-)\Big\rangle - \Big\langle c\big((H^-)^*\big),c(F_+)\Big\rangle = 0.
\]
It turns out that $\psi$ is identically equal to $0$ on $su(3,1)$.
\end{proof}

We would like to characterise the L\'evy process associated to the Sch\"urmann triple $(\rho_{10},c,\psi)$. For this purpose one could compute the action of the Casimir invariants on the vacuum vector $\Omega\in \Gamma\big(L^2(\mathbb{R}_+,D)\big)$.

Since $\psi$ vanishes on $\eufrak{g}$, we have a simple formula for the action of Lie algebra elements on the vacuum vector,
\[
j_{st}(X)\Omega = c(X)\otimes\mathbf{1}_{[s,t]}, \qquad X\in\eufrak{g}_\mathbb{C},\quad 0\le s\le t,
\]
so we have
\begin{gather*}
j_{st}(H_3)\Omega = 0 = j_{st}(H_\pm)\Omega, \\
j_{st}(F_3)\Omega = -\frac{1}{\sqrt{2}} \xi_{10}\otimes \mathbf{1}_{[s,t]}, \quad j_{st}(F_\pm)\Omega = \pm \xi_{1,\pm1}\otimes\mathbf{1}_{[s,t]}.
\end{gather*}
For the Casimir invariants we get
\begin{gather*}
j_{st}(J_2)\Omega = j_{st}(H_+) \big(- \xi_{1,-1}\otimes \mathbf{1}_{[s,t]}\big) + j_{st}(H_-) \big(\xi_{1,1}\otimes \mathbf{1}_{[s,t]}\big) = 0.
\end{gather*}
and
\begin{gather*}
j_{st}(J_1) \Omega = -2 j_{st}(F_3) \left(-\frac{1}{\sqrt{2}}\xi_{10}\otimes \mathbf{1}_{[s,t]}\right)
- j_{st}(F_+)\big(-\xi_{1,-1}\otimes\mathbf{1}_{[s,t]}\big) \\
-j_{st}(F_-)\big(\xi_{11}\otimes \mathbf{1}_{[s,t]}\big) \\
= -(t-s) \Omega + \sqrt{2} \big(\xi_{1,1}\otimes\mathbf{1}_{[s,t]}\big)\otimes \big(\xi_{1,-1}\otimes\mathbf{1}_{[s,t]}\big) \\
+\sqrt{2} \big(\xi_{1,-1}\otimes\mathbf{1}_{[s,t]}\big) \otimes \big(\xi_{1,1}\otimes \mathbf{1}_{[s,t]}\big) \\
- \sqrt{2} \big(\xi_{10}\otimes \mathbf{1}_{[s,t]}\big) \otimes \big(\xi_{10}\otimes \mathbf{1}_{[s,t]}\big).
\end{gather*}.

The action of $j_{st}(J_1)$ on the vacuum vector shows that $j_{st}(J_1)$ is not a multiple of the identity, which implies that the representatoin $j_{st}$ restricted to the subspace generated from the vacuum vector can not be irreducible.

To get a better understanding of the representations $j_{st}$, $0\le s\le t$, we will now consider the restrictions of $\rho_{10}$ to Lie subalgebras of $so(3,1)$.

\section{Restriction to the Lie sub algebra $so(3)$}\label{sec-restr-so3}

The basis we used to describe the representations of $so(3,1)$ is already adapted to the subalgebra $so(3)={\rm span}\{H_3,H_+,H_-\}$, so it is easy to decompose the restriction of representations of $so(3,1)$ to its Lie subalgebra $so(3)$ into its irreducible parts.

The representation $\rho_{10}$ restricted to $so(3)={\rm span}\{H_3,H_+,H_+\}$ decomposes into a direct sum of finite-dimensional irreducible representations. Recall that the irreducible representations of $so(3)$ are all unitarily equivalent to one of the following. Let $s\in\frac{1}{2}\mathbb{Z}$, set $E_s={\rm span}\{e_{-s}, e_{-s+1},\ldots, e_s\}$, where $e_{-s},\ldots,e_s$ form an orthonormal basis, and set
\[
\pi_s(H_3)e_m = m e_m,\quad \pi_s(H_{\pm})e_m= \sqrt{(s\mp m)(s\pm s+1)}e_{m\pm 1},
\]
for $m=-s,\ldots,s$. It is not difficult to check that we have
\[
(D_{10},\rho_{10}|_{so(3)})\cong \bigoplus_{s=3}^\infty (E_s,\pi_s).
\]

\section{Restriction to Lie sub algebra $so(2,1)$}\label{sec-restr-so21}

The basis elements $H_3,F_+,F_-$ span a Lie sub algebra of $so(3,1)$ that is isomorphic to the non-compact form $sl(2;\mathbb{R})\cong su(1,1)\cong so(2,1)$ of the three-dimensional simple Lie algebra $sl(2)$. We will now describe the restriction of our L\'evy processes on $so(3,1)$ to this Lie sub algebra.

Recall that $so(2,1)$ admits the highest and lowest weight representations $\pi^+_t$ and $\pi^-_t$, with $t>0$, acting on $D^\pm_t={\rm span}\{e_n;n\in\mathbb{N}\}$ (where $(e_n)$ are an orthonormal basis) as
\begin{gather*}
\pi^\pm_t (H_3) e_n =  \pm (n+t)e_n, \\
\pi^+_t (F_+) e_n = \sqrt{(n+1)(n+2t)} e_{n+1}, \quad \pi^-_t(F_+) e_n = \sqrt{n(n+2t-1)} e_{n-1} \\
\pi^+_t (F_-) e_n = \sqrt{n(n+2t-1)} e_{n-1}, \quad \pi^-_t(F_-) e_n = \sqrt{(n+1)(n+2t)} e_{n+1},
\end{gather*}
cf.\ \cite{accardi+franz+skeide02}. There is also a third family $\pi_{c,\mu}$ acting on $D^0_{c,\mu}={\rm span}\{f_n,n\in\mathbb{Z}\}$ (where $(f_n)$ is an orthonormal basis) as
\begin{eqnarray*}
\pi_{c,\mu}(H_3) f_n &=&  (n-\mu) f_n, \\
\pi_{c,\mu}(F_+) f_n &=& \sqrt{n^2+(1-2\mu)n + \mu(\mu-1) - c} f_{n+1}, \\
\pi_{c,\mu}(F_-) f_n &=& \sqrt{(n-1)^2+(1-2\mu)(n-1) +\mu(\mu-1)-c }f_{n-1},
\end{eqnarray*}
with $0\le\mu<1$ and $c<\mu(\mu-1)$.

The families $\pi^+_t$ and $\pi^-_t$ are called the positive and the negative discrete series. Our third family contains both the principal unitary series and the complementary unitary series. See, e.g., \cite{vilenkin+klimyk91}, Section 6.4] for more information on the representation theory of $SU(1,1)$.

Denote by $K=H_3^2-\frac{1}{2}(F_+F_-+F_-F_+)= H_3(H_3-1) -  F_-F_+$ the Casimir element of $so(2,1)$. Then we have
\[
\pi^\pm_t(K) = 2t(t - 1)\,{\rm id} \qquad \mbox{ and }\qquad \pi_{c,\mu}(K)=2c\,{\rm id}.
\]

The subrepresentations $\pi^+_t$ and $\pi^-_t$ can be detected by their cyclic vector $e_0$ which is characterised by the equations
\begin{gather*}
\pi^+_t(F_-)e_0=0,\, \pi^+_t(H_3)e_0=te_0 \\
(\mbox{or } \pi^-_t(F_+)e_0=0,\,\pi^+_t(H_3)e_0=-te_0\, \mbox{ resp.}).
\end{gather*}

\begin{proposition}\label{prop-decomp-so21}
If we restrict the representation $\rho_{10}$ of $so(3,1)$ to the Lie subalgebra $so(2,1)$, then it decomposes as
\[
(D_{10},\rho_{10}|_{so(2,1)})\cong (D^+_1,\pi^+_1)\oplus (D^-_1,\pi^-_1)\oplus (D_R,\pi_R)
\]
where the "rest" $(\pi_R,D_R)$ is a direct sum of unitary irreducible representations $(\pi_{c,0},D^0_{c,0})$ belonging to the third family.
\end{proposition}

\begin{proof}
We need to determine all eigenvectors of $\rho_{10}(H_3)$ that are annihilated by $\rho_{10}(F_-)$. A non-zero vector $\xi=\sum_{\ell=1}^\infty\sum_{m=-\ell}^\ell x_{\ell m}\xi_{\ell m}$ is an eigenvector of $\rho_{10}(H_3)$, if and only if there exists an integer $m_0\in\mathbb{Z}$ s.t.\ $x_{\ell m}=0$ for $m\not=m_0$. And this integer is then its eigenvalue.

Let us now study, when such a vector is annihilated by $\rho_{10}(F_-)$. We consider first $m_0\in\{-1,0,1\}$. For $\xi=\sum_{\ell=1}^\infty x_{\ell}\xi_{\ell m_0}\in D_{10}$ we have
\begin{gather}\label{eq-condf-}
\rho_{10}(F_-)\xi = - \sum_{\ell =2}^\infty \left(x_{\ell+1}C_{\ell+1}\sqrt{(\ell+m_0+1)(\ell+m_0)}\right. \\
\qquad +\left. x_{\ell-1}C_\ell\sqrt{(\ell-m_0)(\ell-m_0+1)}\right)\xi_{\ell m_0},
\end{gather}
If this vector vanishes, then the coefficients $x_{2\ell+1}$, $\ell\in\mathbb{N}$, are determined by $x_1$ via the recurrence relation
\begin{equation}\label{eq-recurr}
x_{\ell+1}=-\frac{C_\ell \sqrt{(\ell-m_0)(\ell-m_0+1)}}{C_{\ell+1}\sqrt{(\ell+m_0+1)(\ell+m_0)}} x_{\ell-1}.
\end{equation}
We get
\begin{eqnarray*}
\frac{|x_{\ell+1}|}{|x_{\ell-1}|} &=& \frac{(\ell^2-1)\big(4(\ell+1)^2-1\big)(\ell-m_0)(\ell-m_0+1)}{(4\ell^2-1)\big((\ell+1)^2-1\big)(\ell+m_0+1)(\ell+m_0)} \\
&=& 1-\frac{8m_0}{2\ell} + o\left(\frac{1}{\ell}\right)
\end{eqnarray*}
and the Raabe-Duhamel test shows that for $x_1\not=0$ this series can only converge if $m_0=1$.

Furthermore, in the case $m_0=1$ we get $x_2=0$ from the coefficient of $\xi_{10}$ in \eqref{eq-condf-}, and then $x_{2\ell}=0$ for all $\ell>1$ from the recurrrence relations \eqref{eq-recurr}.

We set $x_1=1$, $x_2=0$ and let $(x_\ell)_{\ell\ge 1}$ denote the solution of the recurrence relation \eqref{eq-recurr}. Then $\xi^+=\sum_{\ell=1}^\infty x_\ell \xi_{\ell 1}$ is a non-zero vector s.t.
\[
\rho_{10}(H_3)\xi^+=\xi^+,\quad \rho_{10}(F_-)\xi^+=0,
\]
it is therefore the cyclic vector of a subrespresentation of $(D_{10},\rho_{10}|_{so(2,1)})$ that is unitarily equivalent to $(D_1^+,\pi_1^+)$.

A careful study of the equation $\rho_{10}(F_-)\xi=0$ shows that all solutions are of the form $\lambda \xi^+$ for some $\lambda\in\mathbb{C}$.

Indeed, the discussion above shows this already for $m_0\in\{-1,0,1\}$. Furthermore, there are no solution with $m_0>1$, because the condition $\rho_{10}(F_-)\xi=0$ implies immediately $x_1=0=x_2$. And the Raabe-Duhamel test allows to show that there are no solutions with $m_0<1$, either.

The discussion of the condition $\rho_{10}(F_+)\xi=0$, which leads to subrepresentations that are unitarily equivalent to a representation of the form $(D_t^-,\pi_t^-)$ is similary. Set $y_1=1$, $y_2=0$ and let $(y_, \ell)_{\ell\ge 1}$ be the sequence determined from these values via the recurrence relation
\[
y_{\ell+1}= \frac{C_\ell \sqrt{(\ell+m)(\ell +m+1)}}{C_{\ell+1}\sqrt{(\ell-m)(\ell-m+1)}}
\]
with $m=1$. Set $\xi^-=\sum_{\ell=1}^\infty y_\ell \xi_{\ell,-1}$. Then we have
\[
\{\xi\in D_{10}; \rho_{10}(F_+)\xi=0\}=\mathbb{C}\xi^-,
\]
which implies that $(D_{10},\rho_{10}|_{so(2,1)})$ contains a unique subrepresentation that is unitarily equivalent to $(D_{-1}^+,\pi_{-1}^+)$.

Since the spectrum of $\rho_{10}(H_3)$ is equal to $\mathbb{Z}$, it follows that the remaining subrepresentations have to belong to the family $(D^0_{c,0},\pi_{c,0})$, $c<0$. 
\end{proof}

\section{Conclusion}

We have identified the Sch\"urmann triple underlying the factorizable current representations of the Lorentz group in \cite{graev+vershik05,vershik+graev06}.

The decomposition in Proposition \ref{prop-decomp-so21} can be used to compute the classical distribution of elements of $j_{st}(F_++F_-+\lambda H_3)$, since the distributions of elements of this form are known for the irreducible unitary representations of $so(2,1)$, cf.\ \cite{koelink+vanderjeugt} and \cite{accardi+franz+skeide02}, and since the direct sums at the level of the Sch\"urmann triple translate into tensor products for the associated L\'evy processes.

It would be interesting to extend these results to the higher rank groups $O(n,1)$ and $U(n,1)$ in the future.

\section*{Acknowledgements}

UF was supported by the French `Investissements d’Avenir' program, project ISITE-BFC (contract ANR-15-IDEX-03), and by an ANR project (No./ ANR-19-CE40-0002).

This work was presented at the ``International Conference on Infinite Dimensional Analysis, Quantum Probability and Related Topics, QP38'' held at Tokyo University of Science in October 2017, and has been submitted for publication in the proceedings of that meeting. We thank the organizers Noboru Watanabe and Si Si for their hospitality.

\end{document}